\documentclass[11pt]{article}

 \usepackage{amsfonts,amssymb,graphicx,amsmath, amsthm,mathdots}
\usepackage[colorlinks=true,linkcolor=blue,citecolor=blue]{hyperref}

\newtheorem{rema}{Remark}

\newtheorem{lemm}{Lemma}
\newtheorem{theo}{Theorem}
\newtheorem{coro}{Corollary}

\newcommand{\C}[1][]{\ensuremath{{\mathbb{C}^{#1}} }}
\newcommand{\R}[1][]{\ensuremath{{\mathbb{R}^{#1}} }}
\renewcommand{\S}[1][]{\ensuremath{{\mathbb{S}^{#1}} }}
\renewcommand{\H}[1][]{\ensuremath{{\mathbb{H}^{#1}} }}
\newcommand{\T}[1][]{\ensuremath{{\mathbb{T}^{#1}} }}

\newcommand{\M}{{\cal M}}

\newcommand{\s}{{\cal S}}
\newcommand{\<}{\langle}
\renewcommand{\>}{\rangle}

\newcommand{\pa}{\partial}

\newcommand{\al}{\alpha}
\newcommand{\eps}{\epsilon}

\newcommand{\ka}{\kappa}
\newcommand{\la}{\lambda}

\newcommand{\si}{\sigma}

\renewcommand{\ell}{\texttt{L}}
\date{}

\title{Marginally trapped submanifolds in space forms with arbitrary signature}
\author{ Henri Anciaux\footnote{University of S\~ao Paulo; supported by CNPq (PQ 306154/2011-0) and Fapesp (2011/21362-2)}
 \\ \em {\small In memory of Franki Dillen } \em}

\begin{document}

\maketitle

\centerline{\textbf {\large{Abstract}}}

\bigskip

{\small We give explicit representation formulas for marginally trapped submanifolds of co-dimension two 
in  pseudo-Riemannian spaces with arbitrary signature and constant sectional curvature.


\bigskip

\centerline{\small \em 2010 MSC: 53A10, 53C42
\em }


\section*{Introduction}
 Let $({\cal N},g)$ be a pseudo-Riemannian manifold and $\s$ a submanifold of $({\cal N},g)$ with non-degenerate induced metric. We shall say that $\s$ is \em marginally trapped \em if its mean curvature vector is null, i.e.\ $g(\vec{H},\vec{H})$ vanishes.
When $({\cal N},g)$ is a Lorentzian four-manifold and $\s$ is spacelike, the marginally trapped condition has an interpretation in terms of general relativity: it describes the horizon of a black hole (\cite{Pe},\cite{CGP}). The equation $g(\vec{H},\vec{H})=0$  is, nevertheless, interesting in whole generality from the geometric view point, being actually the simplest curvature equation which is purely pseudo-Riemannian: in the Riemannian case this equation implies minimality.

\medskip

In \cite{AG},  marginally trapped submanifolds with co-dimension two have been locally characterized in several simple Lorentzian spaces: the Minkowski space $\R^{n+2}$, the  Lorentzian space forms $d\S^{n+2}$ and $Ad\S^{n+2},$ and the Lorentzian products $\S^{n+1} \times \R$ and $\H^{n+1} \times \R$.
Little has been done about marginally trapped surfaces in the case of a manifold with a non Lorentzian metric. In \cite{Ch}, flat marginally trapped surfaces of $\R^4$ endowed with the neutral metric $dx_1^2+dx_2^2-dx_3^2-dx_4^2$ have been studied, while  Lagrangian marginally trapped surfaces of complex space forms of complex dimension two were characterized in \cite{CD}. Recently,  marginally trapped surfaces of certain spaces of oriented geodesics have been investigated (\cite{GG}).

\medskip

  The purpose of the present paper is to extend the results of \cite{AG} to the case of co-dimension two submanifolds in  constant curvature spaces with arbitrary signature, i.e.\  {\it (i)} in the pseudo-Euclidean space $\R^{n+2}_{p+1}$ equipped with the inner product of signature $(p+1,q+1)$, and \textit{(ii)} in the space form  $\S_{p+1}^{n+2}$ of  signature $(p+1,q+1)$ and sectional curvature $1$ (see next Section for more precise definition and notation).
As in \cite{AG}, we rely on the use of the contact structure of  the 
set of null geodesics of the ambient space. The congruence of null lines which are normal to a submanifold of co-dimension two is a Legendrian submanifold with respect to this contact structure. Conversely, given a nulll line congruence ${\cal L}$ which is Legendrian, there exists an infinite-dimensional family of submanifolds, parametrized by the set of real maps $\tau \in C^2({\cal L})$,  such that the congruence is normal to them. In order to obtain our characterization results, we prove that, given a Legendrian, null line congruence ${\cal L}$, the submanifold parametrized by $\tau$  is marginally trapped if and only if the real map $\tau \in C^2({\cal L})$ is a root of certain polynomial map with coefficients valued in  $C^2({\cal L}).$

\medskip

The paper is organized as follows:  Section \ref{statements} introduces some notation and gives the precise statements of the results; Section \ref{s2} 
gives a characterization of
those submanifolds whose second fundamental tensor is null (Theorem \ref{null2ff}), while Section \ref{s3} provides a local representation formula which is similar to that of \cite{AG} (Theorem \ref{zero}). In Section \ref{s4}, an alternative, more global representation formula is given, under certain maximal rank assumption  (Theorems \ref{one} and \ref{S^{n+2}_{p+1}}). Finally, Section \ref{s5} attempts to enlighten the ideas of this paper by
 providing an interpretation of the general construction in terms of contact geometry, and explains also the relation between Theorems \ref{zero} and \ref{one} in the Lorentzian case.

\medskip

This paper is dedicated to the memory of Franki Dillen, 1963-2013.

\section{Statement of results} \label{statements}
Along the paper we fix three integer numbers  $p,q$ and $n$ such that $p+q = n \geq 1.$
We shall denote by $\R^{n+2}_{p+1}$ the $(n+2)$-dimensional real vector space equipped with the inner product of signature $(p+1,q+1)$  given by 
$$ \<.,.\> = \sum_{i=1}^{p+1}  dx_i^2 - \sum_{i=p+2}^{n+2} dx_i^2 .$$
A non-vanishing vector $v$ of $\R^{n+2}_{p+1}$ is said to be \em null \em if $\<v,v\>=0$. We furthermore introduce the hyperquadric
$$\S_{p+1}^{n+2} := \{ x \in \R^{n+3}_{p+2} |  \,  \<x, x\> =1 \}.$$
The induced metric of $\S_{p+1}^{n+2}$, still denoted by $\<.,.\>$, has  signature $(p+1,q+1)$ and constant sectional curvature  $1$. Conversely,
it is well known (see \cite{Kr}) that a simply connected $(n+2)$-dimensional manifold endowed with a pseudo-Riemannian metric with signature $(p+1,q+1)$ and constant sectional curvature, is, up to isometries and scaling, $\R^{n+2}_{p+1}$ or $\S_{p+1}^{n+2}.$ We shall call these spaces \em pseudo-Riemannian space forms. \em
 
\medskip

We shall be concerned with submanifolds $\Sigma$  of $\R^{n+2}_{p+1}$ and $\S^{n+2}_{p+1}$ with non-degenerate induced metric $g$ and whose normal bundle $N \Sigma$  $(i)$ is two-dimensional (so that $\Sigma$ has dimension $n$), and $(ii)$ has indefinite (Lorentzian) metric (so that the induced metric on $\Sigma$ has signature $(p,q)$).
We recall that the \em second fundamental form \em $h$ of $\Sigma$ is the symmetric tensor
$h :T\Sigma \times T\Sigma  \to N \Sigma$
defined by $h(X,Y):= (D_X Y)^{\perp}$, where $(.)^\perp$ denotes the projection onto the normal space $N\Sigma$ and $D$ is the Levi-Civita connection
of ambient space. If $\nu$ is a normal vector field along $\Sigma$, we define the \em shape operator \em of $\Sigma$ with respect to $\nu$ to be the endormorphism of $T\Sigma$ defined by $A_\nu X = - (D_X \nu)^{\top}$, where $(.)^\top$ denotes the projection onto  $T\Sigma.$
The  relation $\<h(X,Y), \nu\> =\<A_\nu X, Y\>$ shows that the second fundamental form and the shape operator carry the same information.

 The \em mean curvature vector \em $\vec{H}$
of the immersion is the trace of $h$ with respect to the induced metric of $\Sigma$ divided by $n$. Our first result is the characterization of $n$-dimensional submanifolds of space forms with \em  null second fundamental form, \em  i.e.\ such that $\forall X,Y \in T\Sigma,  \, h(X,Y)$ is null:

\begin{theo}   \label{null2ff}
Let $\nu$ be a constant, null vector of $\R^{n+2}_{p+1}$ and $\Sigma$ an $n$-dimensional submanifold with non-degenerate induced metric which is contained in the hyperplane $\nu^{\perp}$.
Then $\Sigma$  has null second fundamental form and is therefore marginally trapped. Moreover,  both the tangent and the normal bundles of $\Sigma$ are flat.

Analogously, let $\nu$ be a constant, null vector of $\R^{n+3}_{p+2}$ and $\Sigma$ an $n$-dimensional submanifold of $\S^{n+2}_{p+1}$ with non-degenerate induced metric which is contained in the hypersurface $ \nu^\perp \cap \S^{n+2}_{p+1}.$
Then $\Sigma$  has null second fundamental form and is therefore marginally trapped. Moreover, $\Sigma$ has constant scalar curvature and flat normal bundle.

Conversely, any  submanifold of $\R^{n+2}_{p+1}$ or $\S^{n+2}_{p+1}$ with null second fundamental form is locally described in this way.
\end{theo}

 Quite surprisingly, the method introduced in \cite{AG} in the Lorentzian case can be used here, in the case
of marginally trapped submanifolds whose second fundamental form is not null, providing local parametrizations:

\begin{theo} \label{zero} 
 Let $\si$ be an immersion of class $C^4$ of an $n$-dimensional manifold $\M$ into  $\R^{n+1}_{p+1}$ (resp.\ $\S^{n+1}_{p+1}$) whose induced metric is non-degenerate and has signature $(p,q)$. Denote by $\nu$ the Gauss map of $\si$, which is therefore $\S^{n}_p$-valued (resp.\ $\S^{n+1}_{p+1}$-valued) ,  by $A=-d\nu$ the corresponding shape operator, and by $\tau_i$  the  roots of the polynomial of degree $n-1$
 $$ P(\tau):=tr (( Id -\tau A)^{-1}) .$$
Then the   immersions 
$\varphi_i : \M \to \R^{n+2}_{p+1} = \R^{n+1}_{p+1} \times \R$ (resp.\ $\S^{n+2}_{p+1} \subset \R^{n+2}_{p+2} \times \R$) defined by
$$\varphi_i= (\si + \tau_i  \nu ,\tau_i)$$  
are marginally trapped.

Conversely, any $n$-dimensional marginally trapped submanifold of  $\R^ {n+2}_{p+1}$ (resp.\ $\S^{n+1}_{p+1}$) whose second fundamental form is not null  is locally congruent to the image of such an immersion.

\end{theo}

\begin{rema} If  the shape operator $A$ of $\si$ is diagonalizable (which is not always the case since the induced metric on $\si$ is not definite)
 the polynomial $P$ takes the following form
$$ P(\tau):=\sum_{i=1}^{p} m_i \prod_{j \neq i}^{p}(\ka_j^{-1}- \tau),$$
where  $ \ka_1 , \dots , \ka_p, \, p \geq 2$ are the $p$ distinct, non-vanishing principal curvatures of $\si$ with multiplicity $m_i$. 
\end{rema}

In order to state the next theorem, we introduce some more  notation: writing $ x = (x',x'') \in \R^{n+2}=\R^{p+1} \times \R^{q+1}$,  where $x' \in \R^{p+1}$ and $x''\in \R^{q+1}$, we introduce the conjugation map $ \overline{(x',x'')}:=(x',-x'')$, 
as well as the $n \times n $ diagonal matrix $\overline{Id}_n$ whose $(p,q)$-block decomposition~is  $\overline{Id}_n=\left(  \begin{array}{cc} Id_p  & 0 \\ 0 & -Id_q \end{array} \right)$.

Since the normal spaces $N\Sigma$ are assumed to be two-dimensional and Lorentzian, the marginally trapped assumption $\< \vec{H}, \vec{H}\>=0$ is equivalent to the fact that $\vec{H}$ is contained in one of the two null lines of $N\Sigma$. We shall call \em mean Gauss map, \em and denote by $\nu=(\nu',\nu'')$, the null vector which is collinear to $\vec{H}$ and normalized by the condition 
$\nu \in \S^p \times \S^q \subset \R^{p+1} \times \R^{q+1}.$ The next two Theorems give a global description of those marginally trapped submanifolds whose mean Gauss map has maximal rank. We observe that this is a generic property and that it is a stronger assumption than requiring the   mean curvature vector $\vec{H}$ to have itself maximal rank.

  \begin{theo}   \label{one}
Let $\Omega$ be an open subset of the universal covering of $\S^p \times \S^q$ and $\si \in C^4(\Omega).$
Denote by $\tau_i$ the roots of the polynomial of degree $n-1$
$$P(\tau) = tr ( (\tau  \, Id_n \, + \si \, \overline{Id}_n \,   + 2 Hess (\si))^{-1} ).$$
 Then the immersions 
$$\begin{array}{cccc}  {\varphi}_i: & \Omega & \to & \R^{n+2}_{p+1} \\
 & \nu & \mapsto &    \tau_i \nu + \si \overline{\nu} + 2 \nabla \si  \end{array}$$
are marginally trapped.

Conversely, any connected, marginally trapped $n$-dimensional submanifold of $\R^{n+2}_{p+1}$ whose mean Gauss map $\nu$ has maximal rank is  the image of 
 such an immersion.
\end{theo}

When $n=2$, the condition of maximal rank on $\nu$ is equivalent to the fact that the second fundamental form is not null. Hence Theorems \ref{null2ff} and \ref{one} provide a complete  characterization of marginally trapped surfaces of $\R^4$ with arbitrary signature. Since the Minkowsi case has already been discussed in \cite{AG},  we detail the case $(p,q)=(1,1)$, i.e.\  of a Lorentzian surface in  $\R^4_2$. Observe first that  $\R^4_2$ is endowed with 
\begin{itemize}
\item[(i)] a natural pseudo-K\"ahler structure, with complex structure $J(x_1,x_2,x_3,x_4):=(-x_2,x_1, -x_4,x_3)$ and symplectic form $\omega=\<J.,.\> = dx_1 \wedge dx_2 - dx_3 \wedge dx_4$;  this  corresponds to the identification of
$\R^4_2 $ with $\C^2$ through the formula $(z_1,z_2)=(x_1+ix_2, x_3+ix_4)$;

\item[(ii)] a natural para-K\"ahler structure, with para-complex\footnote{We refer the reader to  \cite{AMT} or \cite{CFG} for material  about
  para-complex geometry,  also called split geometry. \em} structure $K(x_1,x_2,x_3,x_4):=(x_3,x_4, x_1,x_2)$ and symplectic form $\omega'=\<K.,.\> = dx_1 \wedge dx_3 + dx_2 \wedge dx_4$;  this  corresponds to the identification of
$\R^4_2 $ with $\mathbb D^2,$ where $\mathbb D=\{ a+eb \, | \, (a,b) \in \R^2\}$ is the ring of para-complex numbers,
through the formula $(w_1,w_2)=(x_1+ e x_3, x_2+e x_4)$;
\end{itemize}

\begin{coro}   \label{coro1}
Let $\Omega$ be an open subset of $\R^2$ endowed with the Lorentzian metric $du^2-dv^2$ and $\si \in C^4(\Omega).$ Denote by index $u$ (resp.\ index $v$)  the partial derivative with respect to the variable $u$ (resp.\ $v$).
Then the immersion 
$$\begin{array}{cccc}  {\varphi}: & \Omega & \to & \R^4_2 \simeq \C^2 \\
 &(u,v) & \mapsto &   
 [ (\si -\si_{uu} + \si_{vv}+2i\si_u)e^{iu}, (-\si -\si_{uu} + \si_{vv}-2i\si_v)e^{iv} ]
\end{array}$$
is weakly conformal and its null points are characterized by
$ \si + \si_{uu}+\si_{vv} = \pm 2 \si_{uv}.$
Moreover, away from its null points, ${\varphi}$ is marginally trapped.

Conversely, any connected, marginally trapped surface of $\R^4_2$ whose second fundamental form is not null  is the image of 
 such an immersion.
\end{coro}

In \cite{Ch} and \cite{CD}, marginally trapped surfaces of $\R^4_2$ which are in addition respectively flat and Lagrangian with respect to $\omega$ have been characterized. These additional conditions may be readily seen in terms of the formula given above:
\begin{coro} \label{coro2}
The marginally trapped immersion $\varphi$ of Corollary \ref{coro1} is in addition 
\begin{itemize}
\item[--] flat if and only if $ (\pa_{uu}-\pa_{vv} )\Big((\si+ \si_{uu}+ \si_{vv})^2 - 4 \si_{uv}^2 \Big) =0$;
 \item[--] Lagrangian with respect to the symplectic form $\omega$ if and only if
$\si_u+\si_v+\si_{vvv}-\si_{uuv}-\si_{uvv}+\si_{uuu}=0$;
\end{itemize}
Moreover, there is no marginally trapped surface which is in addition Lagrangian with respect to the symplectic form $\omega'$. 

\end{coro} 

In the next theorem we give a characterization of marginally trapped submanifolds whose mean Gauss map has maximal rank in $\S_{p+1}^{n+2}$:

\begin{theo} \label{S^{n+2}_{p+1}} 
 Let $\si: \M \to \S^{p+1}\times \S^{q}$ be an immersed, oriented hypersurface of class $C^4$ whose induced metric has signature $(p,q)$. Denote by $\nu$ its Gauss map  (hence a $\S^{n+2}_{p+1}$-valued map) and by $A=-d\nu$ the corresponding shape operator.
Denote by $\tau_i$ the roots of the polynomial of degree $n-1$
$$P(\tau) = tr ( ( \tau Id - A)^{-1} ).$$
Then the   immersions
$\varphi_i : \M \to \S^{n+2}_{p+1}$ defined by
$\varphi_i:=  \nu+ \tau_i \si$  
are marginally trapped.

Conversely, any connected,  marginally trapped $n$-dimensional submanifold of  $\S^ {n+2}_{p+1}$ whose mean Gauss map  has maximal rank  is the image of  such an immersion.
\end{theo}

Like in the flat case $\R^4_2$, a marginally trapped surface of $\S^{4}_2$ has either null second fundamental form, or a mean Gauss map with maximal rank. Therefore, Theorems \ref{null2ff} and  \ref{S^{n+2}_{p+1}}  provide a complete characterization in this case. It enjoys, moreover, a more explicit description:

\begin{coro} \label{coro3}
Let $\si$ be an immersion of class $C^4$ of a surface  $\M$ into  $\S^{2}\times \S^1$ with Lorentzian induced metric. Denote by $\nu$ the Gauss map of $\si$ (hence a $\S^4_2$-valued map) and by $H$ the (scalar) mean curvature of $\si$ with respect to $\nu.$
Then the   immersion
$\varphi : \M \to \S^{4}_2$ defined by
$$\varphi=\nu +H  \si $$  
is marginally trapped.

Conversely, any connected marginally trapped surface  of  $\S^ {4}_2$ whose second fundamental form is not null  is the image of 
 such an immersion.
\end{coro}

\section{Submanifolds with null second fundamental form (Proof of Theorem \ref{null2ff}) } \label{s2}

Let $\Sigma$ be an $n$-dimensional submanifold  of $\R^{n+2}_{p+1}$ such that the induced metric on the normal bundle $N\Sigma$ is Lorentzian.
Since the intersection of the light cone of $\R^{n+2}_{p+1}$ with $N \Sigma$ is made of two null lines, there exists a null normal frame,  i.e.\  a pair of normal, null vector fields along $\Sigma$ such that $\<\nu, \nu\>=\<\xi,\xi\>=0$ and $\<\nu,\xi\>=2.$
So, given a normal vector $N$, we have
$$ N = \frac{1}{2} (\<N, \xi\> \nu + \< N, \nu\> \xi ).$$

\begin{lemm} \label{lemm}
The second fundamental form $h$ is collinear to $\nu$ (so in particular it is null) if and only if the mean curvature vector $\vec{H}$ is collinear to $\nu$ and  $\nu$ has rank at most $1.$
\end{lemm}

\begin{proof}
We denote by  $(e_1, \dots ,e_n)$ a local, orthonormal, tangent frame along $\Sigma$ and we set  
 $$ h_{ij}^1 :=\< h( e_i , e_j), \nu \> = -\< d\nu (e_i) ,e_j\>.$$
Then we have, taking into account that $\<d\nu, \nu\>=0,$
$$ d \nu (e_i)= -\sum_{i=1}^n h_{ij}^1 e_j + \frac{1}{2} \<d \nu(e_i), \xi\> \nu.$$
Assume first that $h$ is collinear to $\nu.$ Then clearly its trace $n\vec{H}$ is collinear to $\nu$ as well. Moreover, all the coefficients $h_{ij}^1$ vanish,  so by the equation above, for all $i, \, 1 \leq i \leq n,$ the vector $d\nu(e_i)$ is collinear to $\nu,$ and hence $d\nu$ has rank at most $1$. 

Conversely, if $d\nu$ has rank $1$, then for all pair $(i,j), i \neq j, $ $d\nu(e_i)$ and $d\nu (e_j)$ are proportional. Taking into account the symmetry of the tensor $h_{ij}^1$, an elementary calculation implies that there exist $n+1$  real constants $c, \la_1, \la_2, \dots , \la_n$ such that $h_{ij}^1 = c\la_i \la_j.$
If in addition $\Sigma$ is marginally trapped, i.e.\  $\<n \vec{H},\nu\> =tr \, \, [h_{ij}^1 ]_{1 \leq i,j \leq n}= c \sum_{i=1}^n \la_i^2=0,$  then either $c=0$ or 
$(\la_1, \dots, \la_n)=(0,\dots ,0)$  and in both cases the whole tensor $h_{ij}^1$ vanishes, i.e., $h$ is collinear to $\nu$.
\end{proof}

We come back to the proof of Theorem \ref{null2ff}, observing that under the assumption of the Lemma above, $d\nu$ is collinear to 
$\nu$. 
This implies the existence of a  map $\lambda \in C^1(\Sigma)$ such that $\nu =e^{\lambda} \nu_0,$ where $\nu_0 $ is a constant, null vector of $\R^{n+2}_{p+1}$ or $\R^{n+3}_{p+2}$. 
We conclude that $\Sigma \subset \nu^{\perp}_0.$

\smallskip

We now write the Gauss and the Ricci equations in the flat case:
$$\left. \begin{array}{l} \<R(X,Y)Z,W \>  + \<h(X,Z) ,h(Y,W) \> - \< h(X,W) ,h(Y,Z)\> =0 \\
\\
\< R^\perp (X,Y) \nu ,\xi \> - \< [A_{\nu} , A_{\xi} ] X,Y\>=0 \end{array} \right.$$
If $h$ is collinear to $\nu$, both terms $\<h(X,Z) ,h(Y,W) \>$ and $\< h(X,W) ,h(Y,Z)\>$ vanish, hence the curvature of the tangent bundle vanishes. Moreover, if $h$ is collinear to $\nu$, then $A_{\nu}$ vanishes as well and the  normal bundle is flat.

\smallskip

In the case of $\S^{n+2}_{p+1}$, the Gauss and the Ricci equations become
$$\left. \begin{array}{l}  \<R(X,Y)Z,W \>  + \<h(X,Z) ,h(Y,W) \> - \< h(X,W) ,h(Y,Z)\> = \<X,Z\>\<Y,W\>-\<X,W\>\<Y,Z\>\\
\\
 \< R^\perp (X,Y) \nu ,\xi \> - \< [A_{\nu} , A_{\xi} ] X,Y\>=0.\end{array} \right.$$
Again, if $h$  is collinear to $\nu$,  the terms $\<h(X,Z) ,h(Y,W) \>$ and $\< h(X,W) ,h(Y,Z)\>$ vanish. It follows that the scalar curvature 
of the induced metric is constant and equal to~$1.$ Analogously, the  fact that $h$ is collinear to $\nu$ implies the vanishing of $A_{\nu}$ and therefore the flatness of the normal bundle.


\section{Parametrizations by hypersurfaces (proof of Theorem \ref{zero})} \label{s3}

\subsection{The flat case }

Let ${\varphi}=(\psi,\tau)$ be an immersion of a $n$-dimensional manifold $\M$ into $ \R^{n+2}_{p+1} $ whose induced metric
$\tilde{g}:={\varphi}^*\<.,.\>$ has signature $(p,q)$.  In particular the induced metric on the normal space of  $\bar{\varphi}$ is Lorentzian. Let $\tilde{ \nu}$ be one of the two normalized, null
normal fields along ${\varphi}.$  
Since the discussion is local, there is no loss of generality to assume that, modulo congruence, its last component ${\nu}_{n+3}$  does not vanish, so that we may normalize $\tilde{\nu}=(\nu,1)$.

\begin{lemm} We set  $\si:=\psi - \tau \nu.$ 
 Then the map $(\si, \nu) : \M \to \R^{n+1}_{p+1} \times \S^n_p$ is an immersion. 
\end{lemm}

\begin{proof}
Suppose $(\varphi,\nu)$ is not an immersion, so that there exists a non-vanishing vector $v \in T\M$ such that $(d\varphi(v),d\nu(v))=(0,0).$
Since we have $d\psi=d\si + \tau d\nu + d\tau \nu,$
it follows that 
$$d \varphi(v) =(d\psi(v), d\tau(v))= (d\tau(v) \nu , d\tau(v))= d\tau(v) \tilde{\nu},$$
which is a normal to $\varphi$. However the immersion ${\varphi}$ is pseudo-Riemannian and therefore a vector cannot be tangent and normal at the same time, so we get the required contradiction. \end{proof}

\begin{lemm} \label{l2} We have the following relation:
$$\<d\si, \nu\>=0.$$
\end{lemm}

\begin{proof} Using again that $d\psi= d\si + \tau d\nu + d\tau \nu$ and observing 
 that $\<\nu, d\nu\>=0,$ we have 
$$ 0 =  \<d{\varphi},\tilde{\nu}\> = \<(d\psi, d\tau),(\nu, 1)\>= \<d\psi , \nu \> - d\tau= \<d\si, \nu\>.$$

\end{proof}

 \begin{lemm} \label{translation} Given $\eps >0,$ there exists $t_0 \in (-\eps, \eps)$ such that $\si + t_0 \nu$ is an immersion, and $\nu$ is its Gauss map.
  \end{lemm}
 
\begin{proof}

The claim follows from the fact that the set
$ \{ t \in \R | \,  \si + t \nu \mbox {  is not an immersion} \}$ contains at most $n$ elements.
To see this, observe that given a pair of distinct real numbers $(t,t')$, we have
$$Ker (d\si + t d \nu) \cap Ker  (d\si + t' d \nu) = \{ 0\}$$ 
(otherwise we would have a contradiction with the fact that $(\si,\nu)$ is an immersion).
Hence there cannot be more than $n$ distinct values $t$ such that $Ker (d\si + t d \nu) \neq \{ 0\}.$
The fact that $\nu$ is the Gauss map of $\si + t_0 \nu$ comes from Lemma \ref{l2}:
$$\<d(\si + t_0 \nu),\nu\>=\< d\si,\nu\> + t_0\<d \nu,\nu\>=0.$$
\end{proof}

Lemma \ref{translation} shows that there is no loss of generality in assuming that $\si$ is an immersion: if it is not the case, we may
 translate the immersion $\varphi$ along the vertical direction, setting ${\varphi}_{t_0}:={\varphi} - (0,t_0).$
Of course $\varphi$ is marginally trapped if and only if $\varphi_{t_0}$ is so, and moreover
the vector field $\tilde{\nu}$ is still normal to ${\varphi}_{t_0}.$
Finally, observe that the map $\si_{t_0}: \M \to \R^{n+1}_{p+1}$ associated to ${\varphi}_{t_0}$ is 
$$\si_{t_0} =\psi - (\tau - {t_0}) \nu = \psi - \tau \nu + t_0 \nu = \si + t_0 \nu,$$
hence an immersion.

We now describe the first fundamental form of $\varphi$ and its second fundamental form with respect to $\tilde{ \nu}$, both in terms of the
geometry of the immersion $\si$:

\begin{lemm} \label{geozero}
Denote by $g:=\si^* \<.,.\>$ the metric induced on $\M$ by $\si$ and $A$ the shape operator associated to $ \nu$.

Then the metric 
$\tilde{g}:=\varphi^\ast \<.,.\>$ induced on  $\M$ by $\varphi$ is given by the formula
$$\tilde{g}=g(.,.)-2 \tau g(A.,.)+ \tau^2 g(A.,A.).$$
In particular, the non-degeneracy assumption on $\tilde{g}$ implies that $\tau^{-1}$ is not equal to any principal curvature of $\varphi.$
Moreover,  the second fundamental form of $\varphi$ with respect to $\tilde{ \nu}$ is given by
$$\tilde{h}_{\tilde{ \nu}}:= \<\tilde{h}(.,.), \tilde{\nu}\> = g(.,A.)- \tau  g(A.,A.).$$
\end{lemm}

\begin{proof}
Since $\<d\si, \nu\>=\<d \nu, \nu\>=0$, we have, given $v_1,v_2 \in T\M,$
\begin{eqnarray*}
\tilde{g}(v_1,v_2)&=&\<d\varphi(v_1),d\varphi(v_2)\>\\
&=&\<d\si(v_1),d\si(v_2)\> +\tau \<d\si(v_1),d \nu(v_2)\>+ \tau \<d \nu(v_1),d\si(v_2)\> \\
&&+\tau^2 \<d \nu(v_1),d \nu(v_2)\>
+ d \tau (v_1)d \tau (v_2)\< \nu, \nu\>- d \tau (v_1)d \tau (v_2)\\
&=&g(v_1,v_2)-\tau (g(v_1,Av_2)+g(Av_1,v_2))+ \tau^2 g(Av_1,Av_2)\\
&=& g(v_1,v_2)-2 \tau g(Av_1,v_2)+ \tau^2 g(Av_1,Av_2).
\end{eqnarray*}
We calculate the second fundamental form of $\varphi$ with respect to $\tilde{ \nu}=( \nu,1)$:
\begin{eqnarray*} \tilde{h}_{\tilde{\nu}}&=& -\< d\varphi,d\tilde{\nu}\> \\ &=&-\<d\si+ \tau  d \nu + d\tau  \nu,d \nu\>\\ 
&=&-\<d\si,d \nu\> - \tau  \<d \nu ,d \nu\>\\ &=& g(.,A.) - \tau g(A.,A.).
\end{eqnarray*}
\end{proof}
The proof of Theorem \ref{zero} follows easily: denoting by  $\tilde{A}_{\tilde{\nu}}$ the shape operator of ${\varphi}$ with respect to $\tilde{\nu},$ we have, 
 from Lemma \ref{geozero}:
$$ g( \tilde{A}_{\tilde{\nu}} (  Id - \tau A).,  ( Id - \tau A).) = g(., (Id - \tau A). ).$$
It follows that 
$$ \tilde{A}_{\tilde{\nu}}:= ( Id - \tau A)^{-1}$$ and
that $\vec{H}$ is collinear to $\tilde{\nu}$ if and only if $\tilde{A}_{\tilde{\nu}}$ is trace-free, i.e.\ $\tau$ is the root of the polynomial 
$P(\tau) = tr ( (  Id -\tau  A)^{-1} ).$

\begin{rema} If $\varphi$ is minimal, $\tau=0$ is a root of $P(\tau)$. The corresponding immersion $\varphi=(\si,0)$ is not only marginally trapped, but minimal.
\end{rema}

\subsection{The $\S^{n+2}_{p+1}$ case}

Let $\varphi =(\psi, \tau): \M \to \S^{n+2}_{p+1}$ an immersion such that the induced metric
$\tilde{g}:= \varphi^\ast \<.,.\>$ has signature $(p,q)$.
Let $\tilde{ \nu}$ be one of the two normalized, null
normal fields along $\varphi.$  Since the discussion is local, there is no loss of generality to assume that, modulo congruence, its last component $\nu_{n+3}$  does not vanish, so that we may normalize $\tilde{\nu}=(\nu,1)$.

 We define the  \em null projection \em of $\varphi$ to be $\si:= \psi - {\tau} \nu.$
The fact that $(\nu,1) \in T_{\varphi} \S^{n+2}_{p+1},$ i.e.\
$ 0 = \< (\psi, \tau),(\nu, 1)\> = \<\psi, \nu\> -\tau$, implies that $\<\psi, \nu\> =\tau.$
Hence
\begin{eqnarray*}
\<\si, \si\> &=& \<\psi,\psi\> -2 {\tau}\<\psi, \nu\> + \tau^2 \<\nu,\nu\> \\
&=& \<\psi,\psi\> -\tau^2\\
&=& \<\varphi,\varphi\>\\
&=& 1 ,
\end{eqnarray*}
which shows that $\si$ is $\S^{n+1}_p$-valued.
The proofs of the next two lemmas are omitted, since they are similar to the flat case:

\begin{lemm}\label{leg}
 The map $(\si, \nu) : \M \to \S^{n+1}_{p+1} \times \S^{n+1}_{p+1}$ is an immersion. 
\end{lemm}

\begin{lemm} The following relations hold:
$$ \<\si, \nu\>=0 \quad \mbox{ and } \quad
\<d\si, \nu\>=0.$$
\end{lemm}
Unlike in the flat case, there is no vertical translation in $\S^{n+2}_{p+1}$.  
We may however, up to a arbitrarily small  pertubation, assume that $\si$ is an immersion.

\begin{lemm} Given $\eps >0,$ there exists $\alpha \in (-\eps,\eps)$ and a hyperbolic rotation $R^{\al}$ of angle $\alpha$ such that the null projection $\si^{\al}$ of $\varphi^{\al} := R^{\al} \varphi$ is an immersion.
  \end{lemm}
 
\begin{proof}
Set
$$R^{\al}=\left(\begin{array}{ccc} \cosh \al &  & \sinh \al \\  & Id & \\ \sinh \al &&  \cosh \al \end{array} \right) 
 \in SO(p+2,q+1)$$
and $\varphi^\al := R^{\al} \varphi,$ $\tilde{\nu}^\al:= R^{\al} \tilde{\nu}.$
 Observe that $\tilde{\nu}^{\al}:=(\nu^\al, \nu_{n+3}^\al)$ is not anymore normalized \em a priori, \em since its last component $\nu^{\al}_{n+3} $ is
equal to $\cosh (\al) + \sinh (\al) \nu_1,$ where $\nu_1$ is the first component of the vector $\tilde{\nu}.$

Nevertheless, the null geodesic passing through the point $\varphi^\al$ and directed by the vector $\tilde{\nu}^{\al}$
crosses the slice $d\S^{n+2}_{p+2} \cap \{  x_{n+3}=0\}$ at the point
$$(\si^\al,0):= \left(\psi^\al  - \frac{\tau^\al}{\nu_{n+3}^\al} \nu^\al,0 \right).$$
Clearly, $\si^\al$ is an immersion
 if and only if $R^{-\al}\si^{\al}=\psi - \frac{\tau^\al}{\nu_{n+3}^\al} \nu= \varphi +\left(\tau - \frac{\tau^\al}{\nu_{n+3}^\al}\right)\nu $ is so.  Observe that 
\begin{eqnarray*} 
\tau - \frac{\tau^\al}{\nu_{n+3}^\al}&=& \tau - \frac{\cosh (\al) \tau + \sinh (\al) \psi_1}{\cosh (\al)+ \sinh (\al ) \nu_1}\\
&=& \tanh (\al) (\psi_1 - \tau \nu_1 ) + o ( \al)\\
&=& \si_1 \al+ o ( \al).\end{eqnarray*}
Now, assume that  $R^{-\al}\si^\al$ is not an immersion for $\al \in (-\eps, \eps)$. Hence there exists 
a one-parameter family of unit tangent vectors $v^\al$ such that
$$0= d (R^{-\al}\si^\al)(v^\al) = d\si(v^\al) +  (d \si_1(v^\al) \nu +  \si_1 d\nu(v^\al) )\al + o(\al), 
\, \forall \al \in (-\eps, \eps).$$
Thus
$$ \left\{ \begin{array}{l} d\si (v^\al)=0\\
 d \si_1(v^\al) \nu +  \si_1 d\nu(v^\al) = \si_1 d\nu(v^{\al})=0.
\end{array} \right.$$
By Lemma \ref{leg}, $d\nu(v^{\al})$ and $d\si(v^\al)$ cannot vanish simultaneously, therefore $\si_1$ vanishes.
Repeating the argument with suitable rotations yields that all the other coordinates of $\si$ vanish, a contradiction since $\si \in \S^{n+1}_{p+1}.$
\end{proof}

By the previous lemma we may assume that $\si$ is an immersion. The remainder of the proof  follows the lines of that of the flat case,
in particular Lemma \ref{geozero} still holds.

\section{Parametrization by the mean Gauss map} \label{s4}

\subsection{The flat case (proof of Theorem \ref{one})}
In this section $\Sigma$ denotes a $n$-dimensional submanifold of $\R^{n+2}_{p+1}$ whose induced metric has signature $(p,q)$ and such that the normalized vector $\nu \in \S^p \times \S^q \subset  \R^{n+2}$ has rank $n.$ We may therefore parametrize  $\Sigma$ locally  by $\nu$, i.e. by a map  $\varphi: \Omega \to \R^{n+2}_{p+1}$, where $\Omega$
is an open subset of the universal covering of $\S^p \times \S^q.$
We set $\si(\nu):= \frac{1}{2}\<\varphi(\nu), \nu \>$ and $\tau (\nu):=\frac{1}{2} \<\varphi(\nu), \overline{\nu}\>.$\footnote{The pair $(\si,\tau)$ may be regarded as a generalization of the \em support function \em of a hypersurface.}
\begin{lemm}
$$\varphi = \tau \nu + \si \overline{\nu} + 2 \nabla \si,$$
where $\nabla \si$ is the gradient of $\si$ with respect to the induced metric on $\S^p \times \S^q$, i.e.\, 
$\nabla \si = (\nabla' \si, -\nabla'' \si)$, where $\nabla'$ and $\nabla''$ are the gradient on $\S^p$ and $\S^q$ respectively.
\end{lemm}

\begin{proof}
Since $\nu$ and $\overline{\nu}$ are null and $\<\nu, \overline{\nu}\>=2$, we clearly have
$ \varphi = \tau \nu + \si \overline{\nu} + V,$
where $V \in T_{\nu} (\S^p \times \S^q) = T_{\overline{\nu}} (\S^p \times \S^q )= T_{\nu'} \S^p \times T_{\nu''} \S^q.$
In order to determine $V$, we use the assumption $\<d\varphi, \nu\>=0$. Taking into account that
$$ d\varphi = d\tau \nu + \tau d\nu+ d\si \overline{\nu} + \si d \overline{\nu} + dV$$
and that  $\<\nu,\nu\>$, $\<d\nu,\nu\>$ and $ \<d\overline{\nu}, \nu\>$ vanish, we  have
$$ \<d\varphi, \nu\>=d\si \<\overline{\nu}, {\nu}\> + \< dV ,\nu\>=2 d\si + \< dV, \nu\>.$$
On the other hand, from $0=d(\<V, \nu\>)=\<dV, \nu\>+\<V, d\nu\>$,  we conclude, observing that $d\nu= Id,$
$$ \< V,W\>=\<V, d\nu(W)\>=-\<dV(W), \nu\> = 2d\si(W), \forall W \in T_{\nu}(\S^p\times \S^q),$$
which, by the very definition of the gradient, proves that $V = 2\nabla \si$.
\end{proof}

We now complete the proof of Theorem \ref{one}: define the endomorphism $A_\nu$ on $T_\nu (\S^p \times \S^q)$ by
$$\<A_\nu.,.\> = {h}_{\nu}.$$
Hence, using that $d\nu$ is the identity map of $T_\nu (\S^p \times \S^q)$, we have
$$ \< d\varphi \circ A_\nu .,d\varphi.\> =- \< d\nu.,d\varphi .\>=- \<\Pi., d\varphi .\>,$$
where $\Pi $ denotes the restriction to $ T_\nu (\S^p \times \S^q)$ of the normal projection  $\R^{n+2}_{p+1} \to T_{\varphi(\nu)}\s$.
It follows that
$$d\varphi \circ A_\nu = -\Pi$$
and therefore
$$ A_\nu^{-1} = -\Pi^{-1} \circ d\varphi$$
(the maximal rank assumption on $\nu$ implies that $\Pi$ is one-to-one).
In order to calculate the trace of $A_{\nu}$ we introduce an orthonormal basis $(e_1, \dots, e_n)$ of $T_\nu (\S^p \times \S^q)$, such that
$\<e_i,e_i\> =1$ if $1 \leq i \leq p$ and $\<e_i,e_i\>=-1$ if $p+1 \leq i \leq n.$
We define the coefficients $a_{ij}$ by 
$$ d\varphi(e_i) = \sum_{j=1}^n a_{ij} \Pi e_j.$$
Clearly $$ A_{\nu}^{-1} = [a_{ij}]_{1 \leq i,j \leq n}.$$
In order to determine explicitely the coefficients $a_{ij}$, we calculate
\begin{equation} \label{dphi} d \varphi = d\tau \, \nu  + d\tau \, \nu + d\si \, \overline{\nu} + \si \, d \overline{\nu}+ 2d \nabla \si. \end{equation}
Then, we introduce a null, normal vector field $\xi$ along $\s$  such that $(\nu,\xi)$ is a null frame of $N \Sigma = T \Sigma^{\perp}, $ which is in addition normalized, i.e.\ $\<\nu,\xi\>=2.$
Then the projection of a vector $V$ of $\R^{n+2}_{p+1}$ onto $N\Sigma$ is given by the formula:
$$ \frac{1}{2} \left( \< V, \xi\> \nu + \<V, \nu \> \xi \right).$$ 
It follows that

\begin{equation} \label{Pi} \Pi V = V - \frac{1}{2} \left( \< V, \xi\> \nu + \<V, \nu \> \xi \right). 
\end{equation}
For $1 \leq i \leq p$ we have, using Equation (\ref{dphi}) and observing that $d\nu(e_i)=d\overline{\nu}(e_i) =e_i$, 
$$d\varphi(e_i)= (\tau + \si) e_i + d\tau(e_i) \nu + d\si(e_i) \overline{\nu} + 2 d(\nabla \si)(e_i)$$ 
Using Equation (\ref{Pi}) and the fact that $  \< d(\nabla \si)(e_i),e_j\>  = Hess(\si)(e_i,e_j)$, we conclude that, for  $1 \leq i \leq p$,
$$ a_{ij}= \delta_{ij} (\tau+\si) +2 Hess(\si)(e_i,e_j).$$
Analogously we get, if $p+1 \leq i \leq n$,
$$ a_{ij}= \delta_{ij} (\tau-\si) +2 Hess(\si)(e_i,e_j).$$
The conclusion of the proof of Theorem \ref{one} follows easily.


\subsection{The case $(p,q)=(1,1)$ (proof of Corollaries \ref{coro1} and \ref{coro2})} \label{pq=11}
We use the natural identification $\R^4_2 \simeq \C^2$ and denote by $(u,v)$ the natural coordinates on $\S^1 \times \S^1,$ so that  $\nu:= (e^{iu} , e^{iv}).$ In particular the 
 metric on $\S^1 \times \S^1$ is $du^2-dv^2.$ 
 Hence we have
$$ A_{\nu} = -\left( \begin{array}{cc}  \tau + \si + 2 \si_{uu} & 2 \si_{uv}\\
           -2 \si_{uv} &  \tau -\si -2 \si_{vv} \end{array} \right)^{-1}$$
whose trace is $  \frac{2}{\det A_{\nu}} \left( \tau +\si_{uu} -\si_{vv} \right)$. Hence
 $\varphi$ is marginally trapped if and only if $\tau = \si_{vv} -\si_{uu}.$ 

We now study the induced metric $\varphi^*\<.,.\>$: since 
\begin{eqnarray*} \varphi= \tau  (e^{iu} , e^{iv}) + \si  (e^{iu} , -e^{iv}) +2 (i \si_u e^{iu}, -i\si_v e^{iv})
\end{eqnarray*}
we have
\begin{eqnarray*}\varphi_u&=& [ ((\tau-\si)_u+i(2\si_{uu}+\tau+\si))e^{iu}, ( (\tau-\si)_u-2i\si_{uv})e^{iv} ]\\
\varphi_v&=& [ ((\tau+\si)_v+2i\si_{uv})e^{iu}, ( (\tau+\si)_v+i(-2\si_{vv}+\tau-\si))e^{iv} ].\end{eqnarray*}
By a straightforward calculation  the coefficients of the first fundamental form $\varphi^*\<.,.\>$ are:
\begin{eqnarray*} E &:=& (2\si_{uu}+\tau+\si)^2 - 4 \si_{uv}^2\\
 F&:=& 4\si_{uv}(\si_{uu}-\si_{vv}+2\tau)\\
G&:=&- (2\si_{vv}-\tau+\si)^2 + 4 \si_{uv}^2\end{eqnarray*}
The marginally trapped assumption $\tau=\si_{vv}-\si_{uu}$ implies 
$$E=-G=(2\si_{uu}+\tau+\si)^2 - 4 \si_{uv}^2 =(\si+ \si_{uu}+ \si_{vv})^2 - 4 \si_{uv}^2,$$
and the vanishing of $F$, so that $\varphi$ is weakly conformal  (and conformal whenever $E$ does not vanish).


 It is well known that the induced metric of a surface with isothermic coordinates
is flat if and only if its conformal factor is harmonic. Here, we are dealing with the Lorentzian metric $du^2-dv^2$, whose Laplacian operator is $\pa_{uu}-\pa_{vv}$. Hence the induced metric is flat if and only if
$$ (\pa_{uu}-\pa_{vv})E =  (\pa_{uu}-\pa_{vv})\Big((\si+ \si_{uu}+ \si_{vv})^2 - 4 \si_{uv}^2 \Big).$$

\medskip

\noindent \textbf{Marginally trapped Lagrangian surfaces}

\medskip

\noindent We recall that $J(z_1,z_2)=(iz_1, iz_2)$, so that
$$J \varphi_u=[ - (2\si_{uu} + \tau+\si) + i(\tau-\si)_u ) e^ {iu}, (2\si_{uv} + i(\tau-\si)_u )e^{iv}] .$$
Hence, using the usual formula $\omega = \<J.,.\>$,
\begin{eqnarray*} \omega( \varphi_u, \varphi_v) &=&\<J\varphi_u, \varphi_v\>\\
&=&-(2\si_{uu}+\tau+\si)(\tau+\si)_v + 2\si_{uv}(\tau-f)_u -2\si_{uv}(\tau+\si)_v - (\tau-\si)_u (-2\si_{vv} + \tau-\si)\\
&=&-(\tau+\si)_v (\si+\si_{uu}+\si_{vv}+2\si_{uv})+(\tau-\si)_u(\si+\si_{uu}+\si_{vv}+2\si_{uv}) \\
&=& (\si+\si_{uu}+\si_{vv}+2\si_{uv})(-\si_v-\si_{u}-\si_{vvv}+\si_{uuv}+\si_{uvv}-\si_{uuu}).
\end{eqnarray*} 
The first factor does not vanish except at degenerate points, so  $\varphi$ is Lagrangian with respect to $\omega$ if and only if $\si_v+\si_{u}+\si_{vvv}-\si_{uuv}-\si_{uvv}+\si_{uuu}=0$.

Recalling that the para-complex structure is given by $K(x_1,x_2,x_3,x_4):=(x_3,x_4,x_1,x_2)$, we have
$$ K \varphi_u=[ (\tau-\si)_u- 2i\si_{uv})e^{iv},(\tau-\si)_u+i(2\si_{uu}+ \tau+\si))e^{iu}] $$
and so
\begin{eqnarray*} \omega'( \varphi_u, \varphi_v) &=&\<K\varphi_u, \varphi_v\>\\
&=&\cos(u-v) \left( -4 \si_{uv}^2 - (\tau -\si - 2 \si_{vv})(2 \si_{uu}+\tau + \si) \right) \\
 &=& \cos(u-v) ( - 4 \si_{uv}^2 + (\si + \si_{uu} + \si_{vv})^2\\
&=& \cos(u-v) E.
\end{eqnarray*}
Hence $\varphi$ is Lagrangian with respect to $\omega'$ if and only if the induced metric is totally null, which is incompatible with the marginally trapped assumption.


\subsection{The $\S^{n+2}_{p+1}$ case (proof of Theorem \ref{S^{n+2}_{p+1}})}
Let  ${\varphi}: \M \to \S^{n+2}_{p+1}$ an immersed submanifold of co-dimension two of $\S^{n+2}_{p+1}$. Let $\si$ be a normal, null vector field along $\varphi$ which is  normalized in such way that  $\si \in \S^{p+1} \times \S^q.$ We moreover assume that $\si$ has maximal rank, i.e.\ $\si: \M \to \S^{p+1} \times \S^q$ is an immersed hypersurface.

\begin{lemm}[\cite{GS}] \label{L1}
 There exists a unique pair $(\nu, \tau)$, where $\nu: \M \to  \S^{n+2}_{p+1} $ and $\tau\in C^{2}(\M)$ such that
$$\varphi=\nu + \tau\si. $$
Moreover the map $\nu: \M \to  \S^{n+2}_{p+1}$ is the Gauss map of $\si$, i.e.\ we have $\nu  \in T_\si(\S^{p+1} \times \S^q)$
 and $\<d\si, \nu\>=0$.
\end{lemm}

\begin{proof} 
For an abritrary $\tau \in C^2(\M),$ we have that $\nu:=\varphi - \tau \si \in \S^{n+2}_{p+1}.$  Hence we shall determine $\tau \in C^2(\M)$ by the condition
$\nu \in T_\nu(\S^{p+1}\times\S^q).$ Recalling the decomposition $\R^{n+3} =\R^{p+2} \times \R^{q+1}$, and writing $\nu=(\nu', \nu''), \, \si=(\si',\si'')$ accordingly, this condition amounts to  $\<\nu', \si'\>_{p+2}=0$ and $\<\nu'',\si''\>_{q+1}=0$, where $\<.,.\>_{p+2}$ and $\<.,.\>_{q+1}$ denote the Euclidean inner products of $\R^{p+2}$ and $\R^{q+1}$ respectively. These two equations yield $\tau =\<\nu',\si'\>_{p+1} $ and $\tau=\<\nu'',\si''\>_{q+1}$, which are actually two equivalent requirements since $\<\varphi,\si\>=\<\varphi',\si'\>_{p+2}-\<\varphi'',\si''\>_{q+1}$ vanishes. Therefore $\tau$ is uniquely determined by the condition 
$\nu \in T_\si(\S^{p+1} \times \S^q)=T_{\si'} \S^{p+1} \times \T_{\si''} \S^q .$

It remains to check that $\nu$ is the Gauss map of $\si$. For this purpose we differentiate $\varphi = \nu + \tau \si$ and remember that $\si$ is normal to $\varphi$, so that
$$ 0=\< d\varphi, \si\> = \< d\nu, \si\> + d\tau \<\si, \si\> + \tau \< d\si, \si\> = \<d\nu, \si\>.$$
Hence $\<d\nu, \si\>$ vanishes. Since $0=d(\<\nu, \si\>)=\<d\nu,\si\> + \<\nu, d\si\>$,  we deduce that $\<d\si, \nu\>$ vanishes as well.
\end{proof}

Observe that the Lemma above implies furthermore that the induced metric $g:=\si^* \<.,.\>$ is non-degenerate, since $\si(\M)$  is a hypersurface and admits a unit normal vector field.

\begin{lemm} \label{geo}
Denote by $g:=\si^* \<.,.\>$ the metric induced on $\M$ by $\si$ and $A$ the shape operator associated to $ \nu,$ 
i.e.\ $A (v):=-d\nu(v), \, \forall v \in T\M.$
Then the metric 
$\tilde{g}:=\varphi^\ast \<.,.\>$ induced on  $\M$ by $\varphi$ is given by the formula
$$\tilde{g}=\tau^2 g(.,.)-2 \tau g(A.,.)+  g(A.,A.).$$
In particular, the non-degeneracy assumption on $\tilde{g}$ implies that $\tau$ is not equal to any principal curvature of $\nu.$
Moreover,  the second fundamental form of $\varphi$ with respect to ${ \si}$ is given by
$$h_{ \si}:= \<h(.,.), \si\> = g(A.,.) -\tau  g(.,.).$$
 \end{lemm}

\begin{proof}
Taking into account that $d\varphi = d\nu + d\tau \, \si + \tau d\si, $ we have
$$ \tilde{g}=\<d\varphi ,d\varphi\>=\<d\nu, d\nu\> + 2\tau \<d\nu, d\si\> + \tau^2 \< d\si, d\si\>
= g(A.,A.) - 2\tau g(A.,.) + \tau^2 g(.,.)$$
and
$$ h_\si=-\<d\varphi ,d\si\>=-\<d\nu, d\si\> - \tau \< d\si, d\si\>=  g(A.,.) - \tau g(.,.).$$\end{proof}

The proof of Theorem \ref{S^{n+2}_{p+1}} is now straightforward: if $\varphi$ is marginally trapped, we may assume without loss of generality that its mean curvature vector  $\vec{H}$ is collinear to $\si$. By the maximal rank assumption on $\si$  we may  use Lemmas \ref{L1} and \ref{geo}.
Denote by $\tilde{A}_\si$  the shape operator of $\varphi$ with respect to $\si.$ Then, from Lemma \ref{geo} above, we have:
$$ g( \tilde{A}_{\si} ( \tau Id - A).,  (\tau Id -A).) = g(., (\tau Id - A). ).$$
It follows that 
$$ \tilde{A}_{\si}:= (\tau Id - A)^{-1}$$ and
that $\vec{H}$ is collinear to $\nu$ if and only if $A_{\nu}$ is trace-free, i.e.\ $\tau$ is the root of the polynomial 
$P(\tau) = tr ( ( \tau Id - A)^{-1} ).$

\medskip 

\noindent \textbf{The case $(p,q)=(1,1)$ (proof of Corollary \ref{coro3})}

\smallskip

\noindent It is straightforward that if $M$ is a $2 \times 2$ matrix, then $tr (M^{-1})= (\det M)^{-1} tr M$. Hence
$tr ( ( \tau Id - A)^{-1} ) $ vanishes if and only if so does $tr(\tau Id- A)=2 \tau - tr A$. Hence, $\varphi$ is marginally trapped if and only if $\tau = tr A/2:=H$, the (scalar) mean curvature of the immersion $\si.$
This proves Corollary \ref{coro3}.

\section{Further remarks}  \label{s5}
 


\subsection{Interpretation of the result in terms of contact geometry}

The constructions done in the previous sections come from the natural contact structure enjoyed by the spaces of null geodesics of the ambient spaces, and from the fact that the set of null geodesics which are normal to a submanifold of co-dimension two is  Legendrian with respect to this contact structure.

\medskip

The proof of Theorem \ref{zero} is based on the following fact: let ${\cal U}$ be the dense, open subset of null geodesics of $\R^{n+2}_{p+1}$ that cross the horizontal hyperplane $\{ x_{n+2} =0 \}$ (in the Minkowski case $(p,q)=(n,0)$, all null geodesics cross the horizontal hyperplane). Then the correspondence
 $\{ (\si, 0) + t (\nu,1) \, |\, t \in \R\} \mapsto (\si, \nu)$ defines a
contactomorphism between ${\cal U}$ and the unit tangent bundle $T^1 \R^{n+1}_p$. The canonical contact structure $\al$ of the unit tangent of a pseudo-Riemannian manifold $(\M,g)$ is given by the expression $\al=g(d\si, \nu)$,  where $\nu$ is a unit vector tangent to $\M$ at the point $\si.$ Hence, given an immersion  $x  \mapsto ( \si(x),\nu(x))$ of a $n$-dimensional manifold such that $x \mapsto \si(x)$ is an immersion as well (a generic assumption), the Legendre condition $g(d\si_x, \nu(x))=0$ simply means that $\nu$ is the Gauss map of  $\si$, or, equivalently, $\nu$ is a unit vector field normal to the immersed hypersurface  $\si.$

\medskip

The interpretation of the proof of Theorem \ref{one} in terms of contact geometry is as follows: the space of null geodesics of $\R^{n+2}_{p+1}$ may be identified with space of one-jets 
on $\S^{p} \times \S^q$, i.e.\ the space $T (\S^p \times \S^q ) \times \R$:
to the triple $(\nu, V, z) \in T(\S^p \times \S^q) \times  \R, $  we associate the
null line $\{   V + z \overline{\nu} + t  \nu \,  | \, t \in \R\} \subset \R^{n+2}_{p+1}$.
The natural contact structure on the space of one-jets $T\M \times \R$, where $(\M,g)$ is a pseudo-Riemannian manifold,  is given by $\alpha:= \psi - dz$, where $\psi$ is the \em Liouville \em form\footnote{To be more precise, the Liouville form is canonically defined on the cotangent bundle $T^* \M$ of a differentiable manifold $\M$.  If $\M$ is moreover equipped with a pseudo-Riemannian metric (as it is the case of $\S^p \times \S^q$), we may identify  $T^* \M $ and $T\M$ and therefore speak of a Liouville form on $T\M$.}  or \em tautological \em form on $T\M$. 
Moreover,  a  generic Legendrian immersion in  $T \M \times \R$ is  local a section, and takes the form 
$\nu \mapsto (\nu, \nabla \si (\nu), \si (\nu))$, 
where $\si \in C^2(\M)$ and $\nabla$ denotes the gradient of the metric $g.$ It follows, in the case $\M = \S^p \times \S^q$, that a generic Legendrian congruence of null lines of $\R^{n+2}_{p+1}$ takes the form
$$\nu \mapsto \{   \nabla \si (\nu) + \si (\nu) \overline{\nu} + t  \nu \,  | \, t \in \R\},$$
where $\si \in C^2(\S^p \times \S^q)$. The choice of real function $\tau \in C^2(\S^p \times \S^q)$ determines an $n$-dimensional  submanifold
parametrized by $\nu \mapsto \nabla \si (\nu) + \si (\nu) \overline{\nu} + \tau(\nu)  \nu$, one of whose null normal is $\nu.$

\medskip

Finally, the proof of Theorem \ref{S^{n+2}_{p+1}} comes from the fact, proved in \cite{GS}, that the space of null geodesics of $\S^{n+2}_{p+1}$ can be identified with $T^1 (\S^{p} \times \S^{q})$, the unit tangent bundle of $\S^p \times \S^q$, as follows: to the pair $(\nu,\psi) \in T^1 (\S^p \times \S^q),$ we associate the null line
$\{ \psi + t \nu| \, t \in \R\} \subset \S^{n+2}_{p+1}.$ 

\subsection{Relation between Theorems \ref{zero} and \ref{one} in the case $(p,q)=(n,0)$} \label{lor}

In the Lorentzian case $(p,q)=(n,0)$, it is easy to relate the formulas of  Theorems \ref{zero} and \ref{one}. To avoid confusion, all mathematical quantities
refered in Theorem \ref{zero} (resp.\ in Theorem \ref{one}) are written with an index $2$ (resp.\ $3$).
We start writing $\nu_3=(\nu_2,1) \in \S^n \times \S^0 \simeq \S^n \times \{1,-1\},$ so that
 $\overline{\nu}_3= (\nu_2, -1)$. Hence the main formula of Theorem \ref{one} becomes
 $$\varphi= \left( (\tau_3+ \si_3) \nu_2 +2 \nabla \si_3 , \tau_3- \si_3  \right),$$
where $\si_3 \in C^4(\S^n \times \S^0) \simeq C^4(\S^n)$ and $\tau_3$ depends on the second derivatives of $\si_3$.
Introducing
$$\si_2:=2 \si_3 \, \nu_2 + 2 \nabla \si_3 \, \mbox{ and }  \, \tau_2:= \tau_3- \si_, $$
we obtain
$$\varphi= (\si_2 + \tau_2  \nu_2 ,\tau_2),$$  
which is exactly the main formula of Theorem \ref{zero}. 
Observe that $\<d\si_2 , \nu_2\>$ vanishes, i.e.\ $\nu_2$ is normal to the immersion $\si_2$, which is therefore parametrized by its Gauss map. Moreover, 
$\<\si_2, \nu_2\>=2\si_3$, i.e.\ $2 \si_3$ is   the support function of the immersion $\si_2.$


\end{document}